\documentclass[10pt,draft,twoside]{amsart}
\usepackage{amssymb}
\usepackage{latexsym}
\usepackage{amsfonts}
\usepackage{amsmath}

\DeclareMathOperator{\AGL}{AGL}
\DeclareMathOperator{\Aut}{Aut}
\DeclareMathOperator{\rep}{rep}
\DeclareMathOperator{\Rep}{Rep}

\DeclareMathOperator{\PSL}{PSL}
\DeclareMathOperator{\PGaL}{P\Gamma L}

\DeclareMathOperator{\supp}{supp}

\DeclareMathOperator{\fix}{fix}

\DeclareMathOperator{\Diag}{Diag}

\DeclareMathOperator{\SL}{SL}
\DeclareMathOperator{\ASL}{ASL}

\DeclareMathOperator{\Sym}{Sym}
\DeclareMathOperator{\n}{{\bf{v}}}
\DeclareMathOperator{\tw}{tw}

\newtheorem{theorem}{Theorem}[section]
\newtheorem{proposition}[theorem]{Proposition}
\newtheorem{lemma}[theorem]{Lemma}
\newtheorem{corollary}[theorem]{Corollary}

\theoremstyle{definition}

\newtheorem{remark}[theorem]{Remark}

\newcommand{\diag}{\operatorname{\sf Diag}}

\newcommand{\F}{\mathbb F}

\renewcommand{\leq}{\leqslant}
\renewcommand{\geq}{\geqslant}

\numberwithin{equation}{section}

\begin{document}

\title[Twisted permutation codes] 
      {Twisted permutation codes} 
\author{Neil I. Gillespie, Cheryl E. Praeger and Pablo Spiga}
\address{[Gillespie] Heilbronn Institute for Mathematical Research, School of Mathematics, Howard House, University of Bristol, UK.}
\email{neil.gillespie@bristol.ac.uk}
\address{[Praeger] Centre for Mathematics of Symmetry and Computation
School of Mathematics and Statistics
The University of Western Australia
35 Stirling Highway, Crawley
Western Australia 6009.
Also affiliated with King Abdulaziz University, Jeddah, Saudi Arabia.}
\email{cheryl.praeger@uwa.edu.au}
\address{[Spiga] Dipartimento di Matematica e Applicazioni, University of Milano-Bicocca,
Via Cozzi 55, 20125 Milano, Italy.}
\email{pablo.spiga@unimib.it}

\begin{abstract}  
We introduce \emph{twisted permutation codes}, which are frequency permutation arrays 
analogous to repetition permutation codes, namely, codes obtained from the repetition construction applied
to a permutation code.  In particular, we show that a lower bound for the minimum distance of a twisted permutation code is the minimum distance of a
repetition permutation code.  We give examples where this bound is tight, but more importantly, we
give examples of twisted permutation codes with minimum distance strictly greater than this lower bound.  
\end{abstract}

\thanks{{\it Date:} draft typeset \today
}

\maketitle

\section{Introduction}\label{secintro}
Transmitting digital information using existing electrical infrastructure, known as \emph{powerline communication},
has been proposed as a possible solution to the ``last mile problem'' in telecommunications \cite{hanvinck,stateoftheart}.
Constant composition codes are coding schemes that are particularly well suited to deal with
the extra noise present in powerline communication, while at the same time maintaining
a necessary constant power output \cite{chu1,chu2,chu3,luo}.  Moreover, it is suggested in \cite{chu1} that
\emph{frequency permutation arrays}, a class of constant composition codes, 
are particularly well suited for powerline communication.
A frequency permutation array of length $m=rq$ over an alphabet $Q$ of size $q$ is a code
with the property that in each codeword, every letter from $Q$ appears exactly $r$ times.  They have 
been studied in \cite{sophie2,sophie1}.  
In \cite{diagnt}, the first two authors characterised a family of \emph{neighbour transitive codes} (see Section~\ref{secneitran})
in which frequency permutation arrays play  
a central role.  In the same paper, the \emph{permutation codes} generated by groups in this family were classified,
and by applying a repetition construction to these codes, infinite families of non-trivial neighbour 
transitive frequency permutation arrays were constructed.  In this setting, repeating codewords improves
the minimum distance of the code only by a factor of the number of repetitions.  In this paper
we introduce \emph{twisted permutation codes}, which are
frequency permutation arrays that are generated by groups and are analogous to repeated permutation codes.  
We give examples where the minimum distance is improved by a greater amount than that achieved by the repetition
construction.   

Let $T$ be an abstract group and, as $|Q|=q$, identify the Symmetric group on $Q$ with $S_q$, the Symmetric group on $\{1,\ldots,q\}$.  
We call a group homomorphism $\rho$ from $T$ to $S_q$ a \emph{representation} of $T$
of \emph{degree $q$}.  Given such a representation we define the permutation code $C(T,\rho)$ (see Section \ref{secconst}). If $\alpha$ is 
a codeword in $C(T,\rho)$, we let $\rep_r(\alpha)=(\alpha,\ldots,\alpha)$ denote the
$r$-tuple with constant entry $\alpha$, and we let \begin{equation}\label{repconst}\Rep_r(C(T,\rho))=\{\rep_r(\alpha)\,:\,\alpha\in C(T,\rho)\}.
\end{equation}
The code $\Rep_r(C(T,\rho))$ is a frequency permutation array of length $rq$ where every
letter appears $r$ times in each codeword.  (This is the repetition construction mentioned above.)  
A twisted permutation code is a frequency permutation array generated by a group $T$
and several (not necessarily distinct) representations of $T$ of the same degree. (See also Table \ref{max}.)
Specifically, we consider an ordered $r$-tuple $\mathcal{I}$ of representations 
of $T$ to $S_q$ and construct the twisted permutation code 
$C(T,\mathcal{I})$ (see Section \ref{secconst}), a frequency permutation array of length $rq$ over $Q$.  
By letting $\delta_{\tw}$ be the minimum distance of $C(T,\mathcal{I})$ and 
$\delta_{\rep}$ be the minimum of the minimum distances of $\Rep_r(C(T,\rho))$ as $\rho$ varies over $\mathcal{I}$, we prove the
following.  

\begin{theorem}\label{mainthm}  Let $q$ be a positive integer, $T$ an abstract group, and $\mathcal{I}$ an ordered $r$-tuple of (not necessarily distinct)
permutation representations of $T$ into $S_q$.  Then $C(T,\mathcal{I})$ as defined in Section \ref{secconst}
is a frequency permutation array of length $rq$ with minimum distance $\delta_{\tw}\geq\delta_{\rep}$.  Moreover,
the twisted permutation codes described in Table \ref{tablethm} have a minimum distance that is strictly greater than this lower bound.  
\end{theorem}

\begin{table}[h]
\centering
\begin{tabular}{c|c|c|c|c|c}
\hline\hline
$T$ & $r$ & $q$ & $\delta_{\rep}$ & $\delta_{\tw}$& Ref. \\
\hline
$S_6$&2&6&4&8&Sec. \ref{secs6}\\
$A_6$&2&6&6&8&Sec. \ref{seca6}\\
$\ASL(3,2)$&2&8&8&12&Sec. \ref{secasl}\\
$S_6$&4&60&176--192&$\leq$224&Sec. \ref{secs62}\\
\hline
\end{tabular}
\caption{Examples of Twisted Permutation Codes with improved minimum distance.}\label{tablethm}
\end{table}\label{firsttable}

\begin{remark}The codes in the fourth line of Table \ref{firsttable} are described in Section \ref{secs62}. There are several different repetition 
permutation codes and also several different twisted permutation codes, with minimum distances, which can be obtained explicitly 
from Table \ref{s62} in Section \ref{secs62}, ranging from 176 to 192, and from 176 to 224, respectively.
\end{remark}

In some cases, two representations of a group $T$ to $S_q$ 
can be identified with each other if there exists a relabelling of the point set $Q$ that maps one to the other.
However, this is not always possible, in which case the representations are distinct.  
For example, $S_6$ has two distinct representations of degree $6$, which are interchanged by an outer automorphism of order $2$.  
Each finite $2$-transitive almost simple group has at most two distinct representations of the same degree, with one infinite family
and six exceptional cases that have exactly two distinct representations~\cite[Table 7.4]{pcam}.  (This fact is a consequence of the 
Classification of Finite Simple Groups.)  Additionally, these groups share the property of $S_6$ that
the two actions are interchanged by an outer automorphism of order $2$.  For $T$ being one of these groups and $\rho_1,\rho_2$ the
distinct representations of $T$ of the same degree, we consider the codes $C(T,(\rho_1,\rho_2))$ in Section~\ref{secneitran} where we 
determine their minimum distance with respect to the lower bound $\delta_{\rep}$.  We also prove in 
Theorem \ref{ntthm} that these codes are neighbour transitive. 

Cameron \cite{pcam} also states that the $2$-transitive affine group $\ASL(2,r)$ with $r=2^f$ 
for some positive integer $f\geq 2$ has $r$ distinct representations of the same degree.
In Section \ref{secaff} we give an explicit construction of these distinct representations, and by letting $\mathcal{I}$ be an
$r$-tuple of these actions, we determine the minimum distance of $C(\ASL(2,r),\mathcal{I})$ with respect
to the lower bound $\delta_{\rep}$.  Finally, in Section \ref{seccomp}, we use the computer software program GAP to construct
some further examples of twisted permutation codes.  This allows us to prove Theorem \ref{mainthm} in Section \ref{proofmain}.  

\section{Definitions}\label{secdef}

\subsection{Codes} A code of length $m$ over an alphabet $Q$ of size $q$ can be embedded as a subset of the vertex set of
the Hamming graph $\Gamma=H(m,q)$, which has a vertex set $V(\Gamma)$ that consists
of $m$-tuples with entries from $Q$, and an edge exists between two
vertices if and only if they differ in precisely one entry.  Throughout this
paper we identify the alphabet $Q$ with the set $\{1,\ldots,q\}$ and
the group $\Sym(Q)$ with $S_q$.  The 
automorphism group of $H(m,q)$, which we denote by $\Aut(\Gamma)$, is
a semi-direct product $B\rtimes L$ where $B\cong S_q^m$ and $L\cong
S_m$ \cite[Thm. 9.2.1]{distreg}.  Let $g=(g_1,\ldots, g_m)\in
B$, $\sigma\in L$ and $\alpha=(\alpha_1,\ldots,\alpha_m)\in V(\Gamma)$. 
Then $g\sigma$ acts on $\alpha$ in the following way:        
\begin{align*}
\alpha^{g\sigma}=&(\alpha_{1{\sigma^{-1}}}^{g_{1{\sigma^{-1}}}},\ldots,\alpha_
{m{\sigma^{-1}}}^{g_{m{\sigma^{-1}}}}).
\end{align*} For all pairs of vertices $\alpha,\beta\in V(\Gamma)$, the \emph{Hamming
  distance} between $\alpha$ and $\beta$, denoted by
$d(\alpha,\beta)$, is defined to be the number of entries in which the
two vertices differ.  It is the distance between $\alpha$ and $\beta$ in $\Gamma$.  
We let $\Gamma_k(\alpha)$ denote the set of vertices in $\Gamma$ that are at distance $k$ from $\alpha$. 

The \emph{minimum distance, $\delta(C)$,} of a code $C$ is the smallest distance between distinct
codewords of $C$.    If $C$ consists of exactly one codeword, then we let $\delta(C)=0$.  
Another code $C'$ in $H(m,q)$ is \emph{equivalent} to $C$ if there exists
$x\in\Aut(\Gamma)$ such that $C^x=C'$, and if $C=C'$ we call $x$ an
\emph{automorphism of $C$}.  The \emph{automorphism group of $C$}
is the setwise stabiliser of $C$ in $\Aut(\Gamma)$, which we denote by
$\Aut(C)$.  The \emph{inner distance distribution of $C$} is the 
$(m+1)$-tuple $\kappa(C)=(a_0,\ldots,a_m)$
where \begin{equation}\label{ai}a_i=\frac{|\{(\alpha,\beta)\in 
  C^2\,:\,d(\alpha,\beta)=i\}|}{|C|}.\end{equation}   
We observe that $a_i\geq 0$ for all $i$ and $a_0=1$.  Moreover,
$a_i=0$ for $1\leq i\leq \delta-1$ and $|C|=\sum_{i=0}^ma_i$. 
For a codeword $\alpha$, the \emph{distance distribution from $\alpha$} is
the $(m+1)$-tuple $\kappa(\alpha)=(a_0(\alpha),\ldots,a_m(\alpha))$ where
$a_k(\alpha)=|\Gamma_k(\alpha)\cap C|$.  

We say a code $C$ is \emph{distance invariant} if the number of codewords at distance
$i$ from a codeword is independent of the choice of codeword.  That is $\kappa(C)=\kappa(\alpha)$
for each codeword $\alpha$.  It is straightforward to deduce that if a group of automorphisms of a 
code acts transitively, then the code is necessarily distance invariant.  

\subsection{Permutation Groups} Let $\Omega$ be a non-empty set.
We denote the group of permutations of $\Omega$ by $\Sym(\Omega)$.  A 
\emph{permutation group} on $\Omega$ is a subgroup of $\Sym(\Omega)$.  
Suppose $G$ is a permutation group on $\Omega$ and $t\in G$.  We define
the \emph{support of $t$} as the set $\supp(t)=\{\alpha\in\Omega\,:\,\alpha^t\neq\alpha\}$
and the set of \emph{fixed points of $t$} as $\fix(t)=\{\alpha\in\Omega\,:\,\alpha^t=\alpha\}$.  It 
follows that $\Omega=\supp(t)\cup\fix(t)$ for all $t\in G$.  We say $G$ acts \emph{regularly} on $\Omega$ if $G$ acts
transitively on $\Omega$ and $G_\alpha=1$ for all $\alpha\in\Omega$.  

Let $G$ be an abstract group now.
An \emph{action} of $G$ on $\Omega$ is a homomorphism $\rho$ from $G$ to $\Sym(\Omega)$, in which
case we say \emph{$G$ acts on $\Omega$} or \emph{$\rho$ defines an action of $G$ on $\Omega$}.  We also
call $\rho$ a (\emph{permutation}) \emph{representation} of $G$ on $\Omega$.  The \emph{degree of the
action} is the cardinality of $\Omega$.  In this paper, all actions have finite degree.  
Let $\rho_1:G\longrightarrow\Sym(\Omega)$ and $\rho_2:H\longrightarrow\Sym(\Omega')$ be actions 
of the groups $G,H$ on $\Omega$ and $\Omega'$.  We say these actions are \emph{permutationally isomorphic} if there exists a 
bijection $\lambda:\Omega\longrightarrow\Omega'$ and an
isomorphism $\varphi:\rho_1(G)\longrightarrow \rho_2(H)$ such that 
\begin{equation}\label{acteq}\lambda(\alpha^{\rho_1(g)})=\lambda(\alpha)^{\varphi(\rho_1(g))}\quad\quad
\textnormal{for all $\alpha\in\Omega$ and $g\in G$},\end{equation}
and we call $(\lambda,\varphi)$ a \emph{permutational isomorphism}.  
If $G=H$ and $\varphi$ is the identity map, then we say the two actions of $G$ are \emph{equivalent}.  However, 
if there does not exist a bijection $\lambda:\Omega\longrightarrow\Omega'$ such that \eqref{acteq} holds with
$G=H$ and $\varphi$ equal to the identity map, then we say the two actions are \emph{inequivalent.}

\section{Constructions}\label{secconst}

Let $Q=\{1,\ldots,q\}$ and $H(q,q)$ be the Hamming graph of length $q$ over $Q$.  We denote 
the Symmetric group on $Q$ by $S_q$.  Let $T$
be an abstract group and $\rho:T\longrightarrow S_q$ be an action of $T$ on $Q$ denoted by $t\mapsto t\rho$.  
For $t\in T$, we identify the permutation $t\rho$ with the vertex in $H(q,q)$ that represents its passive form, that is, 
with $\alpha(t,\rho)=(1^{t\rho},\ldots,q^{t\rho})\in H(q,q)$.
We naturally define 
\begin{equation}\label{1code}C(T,\rho)=\{\alpha(t,\rho)\,:\,t\in T\}.\end{equation}  
The code $C(T,\rho)$ is an example of a \emph{permutation code}.  
Permutation codes were introduced in the 1970s \cite{blake,blakeetal,frankl}, where \emph{sets} 
of permutations in their passive form were considered rather than groups.  
Due to applications in \emph{powerline communication}, Chu, 
Colbourn and Dukes \cite{chu1} renewed the interest in permutation codes, giving new constructions
of such codes.  Other interesting results on permutation codes include a beautiful decoding algorithm by Bailey \cite{bailey} for permutation 
codes of groups; Cameron and Wanless' \cite{camwan} examination of the covering radius of a permutation code; and Cameron and Gadouleau's \cite{camgad}
introduction of the \emph{remoteness of a code} and their examination of this parameter with respect 
to permutation codes.

As discussed in Section \ref{secdef}, the automorphism group of $\Gamma=H(q,q)$ is equal to $B\rtimes L$ where $B\cong S_q^q$ and
$L\cong S_q$.  To distinguish between automorphisms of $\Gamma$ and permutations in $S_q$, we introduce the
following notation.  For $t\in T$ and $\rho:T\longrightarrow S_q$, we let 
$x_{t\rho}=(t\rho,\ldots,t\rho)\in B$, and we let $\sigma(t\rho)$ denote the
automorphism induced by $t\rho$ in $L$.  Since $\rho$ is a homomorphism, it holds 
for $t\in T$ and $\alpha(s,\rho)\in V(\Gamma)$ that 
$$\alpha(s,\rho)^{x_{t\rho}}=(1^{s\rho},\ldots,q^{s\rho})^{(t\rho,\ldots,t\rho)}=
(1^{s\rho t\rho},\ldots,q^{s\rho t\rho})=(1^{(st)\rho},\ldots,q^{(st)\rho})=\alpha(st,\rho).$$
Now, suppose that $i^{t\rho}=j$ for $i,j\in Q$.  Then, by considering
$\alpha(s,\rho)$ as the $q$-tuple $(\alpha_1,\ldots,\alpha_q)$, it holds 
that $\alpha(s,\rho)^{\sigma(t\rho)}|_j=\alpha_i=i^{s\rho}=j^{t^{-1}\rho s\rho}=j^{(t^{-1}s)\rho}$.  Thus $\alpha(s,\rho)^{\sigma(t\rho)}=\alpha(t^{-1}s,\rho)$, 
and we have proved the following.       

\begin{lemma}\label{start} Let $\alpha(s,\rho)\in C(T,\rho)$ and $t\in T$.
  Then $\alpha(s,\rho)^{x_{t\rho}}=\alpha(st,\rho)$ and
  $\alpha(s,\rho)^{\sigma(t\rho)}=\alpha(t^{-1}s,\rho)$.
\end{lemma}

As any group has a regular action on itself by right multiplication, it is a consequence of Lemma \ref{start} 
that $\Diag(T,\rho)=\{x_{t\rho}\,:\,t\in T\}$ acts regularly on $C(T,\rho)$.
This, in particular, implies that $C(T,\rho)$ is distance invariant.

\begin{lemma}\label{mindist1} 
For $t\in T$ we have $d(\alpha(1,\rho),\alpha(t,\rho))=|\supp(t\rho)|$.  Moreover, 
$C(T,\rho)$ has minimum distance $\delta(C(T,\rho))=\min\{|\supp(t\rho)|\,:\,1\neq t\in T\}$, the minimal degree of $T\rho$.  
\end{lemma}

\begin{proof}  For $1\neq t\in T$ it follows that  
$\alpha(1,\rho)|_i\neq\alpha(t,\rho)|_i$ if and only if $i\neq i^{t\rho}$, which holds 
if and only if $i\in\supp(t\rho)$, from which the first statement follows. Now, as $C=C(T,\rho)$ is distance invariant, it has minimum distance 
$$\delta(C)=\min\{d(\alpha(1,\rho),\alpha(t,\rho))\,:\,1\neq t\in T\}.$$ 
  Thus $\delta(C)=\min\{|\supp(t\rho)|\,:\,1\neq t\in T\}$.
\end{proof}

We now consider a more general construction.  Let $\mathcal{I}=(\rho_1,\ldots,\rho_r)$ be an ordered list of 
$r$ (not necessarily distinct) representations from $T$ to $S_q$ and define 
$$\alpha(t,\mathcal{I})=(\alpha(t,\rho_1),\ldots,\alpha(t,\rho_r))\in H(rq,q),$$
which is an $r$-tuple of codewords of the form given in (\ref{1code}).  Hence, we naturally define 
$$C(T,\mathcal{I})=\{\alpha(t,\mathcal{I})\,:\,t\in T\}.$$ We call 
$C(T,\mathcal{I})$ a \emph{twisted permutation code}.  Note that if $r=1$ this is just
the construction given in (\ref{1code}), and if $\rho_1=\cdots=\rho_r$ then $C(T,\mathcal{I})=\Rep_r(C(T,\rho_1))$ as in \eqref{repconst}.

\begin{proposition}\label{distlowerbound} Consider the code $C(T,\mathcal{I})$, with notation as above.  
Then $C(T,\mathcal{I})$ is a frequency permutation array of length $rq$.  Moreover \begin{itemize}
\item[(i)] there exists a group of automorphisms acting regularly on $C(T,\mathcal{I})$.  In particular $C(T,\mathcal{I})$ is distance invariant;
\item[(ii)] the size of $C(T,\mathcal{I})$ is equal to the order of the factor group $T/K$, where $K=\cap_{\rho\in\mathcal{I}}\ker\rho$;
\item[(iii)] $\delta(C(T,\mathcal{I}))=\min_{t\in T^\#}\sum_{\rho\in\mathcal{I}}|\supp(t\rho)|\geq \min_{\rho\in\mathcal{I}}\{\delta(\Rep_r(C(T,\rho)))\}$, where
$T^\#=T\backslash\{1\}$.  
\end{itemize}
\end{proposition}

\begin{proof} Each codeword in $C(T,\mathcal{I})$ is an $r$-tuple of permutation codewords, so it is clear that
$C(T,\mathcal{I})$ is a frequency permutation array of length $rq$.  

(i) For $t\in T$ let $x(t,\mathcal{I})=(x_{t\rho_1},\ldots,x_{t\rho_r})\in\Diag(T,\rho_1)\times\cdots\times\Diag(T,\rho_r)$, and let
$\Diag(T,\mathcal{I})=\{x(t,\mathcal{I})\,:\,t\in T\}$.  Then $x(t,\mathcal{I})$
acts naturally on $\alpha(s,\mathcal{I})$ in the following way:  
\begin{align*}
\alpha(s,\mathcal{I})^{x(t,\mathcal{I})}&=(\alpha(s,\rho_1),\ldots,\alpha(s,\rho_r))^{(x_{t\rho_1},\ldots,x_{t\rho_r})}\\
&=(\alpha(s,\rho_1)^{x_{t\rho_1}},\ldots,\alpha(s,\rho_r)^{x_{t\rho_r}})\\
&=(\alpha(st,\rho_1),\ldots,\alpha(st,\rho_r))&\textnormal{(by Lemma \ref{start})}\\
&=\alpha(st,\mathcal{I}).
\end{align*}
As $T$ has a regular action on itself by right multiplication, we deduce that $\Diag(T,\mathcal{I})$ acts
regularly on $C(T,\mathcal{I})$.  Hence $C(T,\mathcal{I})$ is distance invariant.  

(ii)  From the proof of (i) it follows that $\alpha(s,\mathcal{I})=\alpha(t,\mathcal{I})$ 
if and only if $\alpha(st^{-1},\mathcal{I})=\alpha(1,\mathcal{I})$,
which holds if and only if $st^{-1}\in K$.  Hence the size of $C(T,\mathcal{I})$ is equal to the order of
the factor group $T/K$.  

(iii) For $t\in T^\#$ it is clear that the distance between $\alpha(1,\mathcal{I})$ and
$\alpha(t,\mathcal{I})$ in $H(rq,q)$ is equal to the sum of the distances between $\alpha(1,\rho)$ and
$\alpha(t,\rho)$ in $H(q,q)$ as $\rho$ varies over $\mathcal{I}$.  That is, 
\begin{equation}\label{general1}d(\alpha(1,\mathcal{I}),\alpha(t,\mathcal{I}))=\sum_{\rho\in\mathcal{I}}d(\alpha(1,\rho),\alpha(t,\rho))
=\sum_{\rho\in\mathcal{I}}|\supp(t\rho)|,\end{equation}
where the last equality follows from Lemma \ref{mindist1}.  Thus the distance between
$\alpha(1,\mathcal{I})$ and any codeword in $C(T,\mathcal{I})$ is minimised when this expression is minimised.  Consequently,
as $C(T,\mathcal{I})$ is distance invariant, it follows that $$\delta(C(T,\mathcal{I}))=\min_{t\in T^\#}\sum_{\rho\in\mathcal{I}}|\supp(t\rho)|.$$
To prove the inequality in the statement, we make the following observation.  
\begin{align*} \min_{t\in T^\#}\sum_{\rho\in\mathcal{I}}|\supp(t\rho)| &\geq \min_{t\in T^\#}\{r\cdot\min_{\rho\in\mathcal{I}}\{|\supp(t\rho)|\}\}\\
&=\min_{\rho\in\mathcal{I}}\{r\cdot \min_{t\in T^\#}\{|\supp(t\rho)|\}\}\\
&=\min_{\rho\in\mathcal{I}}\{r\cdot \delta(C(T,\rho))\}&\textnormal{(by Lemma \ref{mindist1})}\\
&=\min_{\rho\in\mathcal{I}}\{\delta(\Rep_r(C(T,\rho)))\}.
\end{align*}
\end{proof}

\begin{remark} (a) Consider the code $C(T,\mathcal{I})$ and let $K$ be as in Proposition \ref{distlowerbound} (ii).
Also let $\tilde{T}=T/K$, and for $\rho\in\mathcal{I}$ define $\tilde{\rho}:\tilde{T}\longrightarrow S_q$
given by $Kt\longmapsto t\rho$.  It is straightforward to check that $\tilde{\rho}$ is well defined and that $\ker\tilde{\rho}=\ker\rho/K$.  
By defining $\tilde{\mathcal{I}}=(\tilde{\rho}_1,\ldots,\tilde{\rho}_r)$, it follows
that $C(T,\mathcal{I})=C(\tilde{T},\tilde{\mathcal{I}})$.  Moreover, 
$$\tilde{K}=\cap_{\tilde{\rho}\in\tilde{\mathcal{I}}}\ker\tilde{\rho}=\cap_{\rho\in\mathcal{I}}(\ker\rho/K)=(\cap_{\rho\in\mathcal{I}}\ker\rho)/K=1.$$
Thus, for any twisted permutation code, by replacing $T$ with $T/K$ we can assume
that $K=1$ and that $|C(T,\mathcal{I})|=|T|$.

(b) The lower bound in Proposition \ref{distlowerbound} (iii) can be equal to zero.  For example, if for 
some representation $\rho'\in\mathcal{I}$ it holds that $\ker\rho'=T$ then 
$\min_{\rho\in\mathcal{I}}\{\delta(\Rep_r(C(T,\rho)))\}=0$.  This is because $\Rep_r(C(T,\rho'))$
consists of just one codeword.\end{remark}

In Sections \ref{secex} and \ref{seccomp}, we give examples of twisted permutation codes with minimum
distance strictly greater than the lower bound in Proposition \ref{distlowerbound} (iii).
However, this lower bound can be attained by letting $\mathcal{I}=(\rho,\ldots,\rho)$ for some representation $\rho:T\longrightarrow S_q$, 
because as we said above, in this case $C(T,\mathcal{I})=\Rep_r(C(T,\rho))$.  
The following result shows that this lower bound can also be attained in a slightly more general setting.  

\begin{lemma}\label{diffreps}  Let $\mathcal{I}=(\rho_1,\ldots,\rho_r)$ be an $r$-tuple of actions of $T$ of degree $q$.  
Suppose that $|\supp(t\rho_i)|=|\supp(t\rho_j)|$
for all $i,j$ and for all $t\in T$.  Then $C(T,\mathcal{I})$ has the same inner distance distribution as $\Rep_r(C(T,\rho_i))$
for $i=1,\ldots,r$.  In particular $\delta(C(T,\mathcal{I}))$ achieves the lower bound in Proposition \ref{distlowerbound} (iii).  
\end{lemma}

\begin{proof}  Let $C=C(T,\mathcal{I})$.  By Proposition \ref{distlowerbound}, $C$ is distance invariant.  
Thus the inner distance distribution of $C$ is equal to the distance distribution from $\alpha(1,\mathcal{I})$.  That is, 
the $k$th entry of $\kappa(C)$ is equal to $|\Gamma_k(\alpha(1,\mathcal{I}))\cap C|$.  It follows
from (\ref{general1}) that  
\begin{equation}\label{som1}\alpha(t,\mathcal{I})\in\Gamma_k(\alpha(1,\mathcal{I}))\cap C \iff \sum_{\rho\in\mathcal{I}} |\supp(t\rho)|=k.\end{equation}
As $|\supp(t\rho_i)|=|\supp(t\rho_j)|$ for all $i,j$, the expression on the right of \eqref{som1} 
is equal to $r|\supp(t\rho_i)|$ for each $i=1,\ldots,r$.  Thus
the $k$th entry of $\kappa(C)$ is equal to $|\{t\in T\,:\,|\supp(t\rho_i)|=k/r\}|$ for each~$i=1,\ldots,r$.  

Now, for $i\in\{1,\ldots,r\}$ let $\mathcal{I}_{\rho_i}=(\rho_i,\ldots,\rho_i)$, and let $C'=C(T,\mathcal{I}_{\rho_i})=\Rep_r(C(T,\rho_i))$.
Again, because $C'$ is distance invariant, the $k$th entry of $\kappa(C')$ 
is equal to $|\Gamma_k(\alpha(1,\mathcal{I}_{\rho_i}))\cap C'|$ and 
$$\alpha(t,\mathcal{I}_{\rho_i})\in\Gamma_k(\alpha(1,\mathcal{I}_{\rho_i}))\cap C' 
\iff \sum_{j=1}^r |\supp(t\rho_i)|=r|\supp(t\rho_i)|=k.$$
Hence the $k$th entry of $\kappa(C')$ is equal to $|\{t\in T\,:\,|\supp(t\rho_i)|=k/r\}|$, as above.  
\end{proof}

\noindent In Section \ref{secaff} we give an example of an infinite family of twisted permutation codes,
each generated by a set $\mathcal{I}$ of $r$ representations that are pairwise distinct, and 
whose minimum distance achieves the lower bound in Proposition \ref{distlowerbound} (iii).  
 
\section{Examples of Twisted Permutation Codes}\label{secex}

The first examples of twisted permutation codes that we introduce 
are constructed from finite $2$-transitive groups of almost simple type.  It is well known
that a finite $2$-transitive group of almost simple type has at most $2$ inequivalent actions,
and the groups with precisely $2$ actions are listed in Table \ref{table1}, which is taken from \cite[Table 7.4]{pcam}.
We consider each group $T$ from Table \ref{table1} as a permutation group in its natural action, so $q$ is equal
to the degree of $T$.  For each line in Table \ref{table1} it holds that the normaliser in $S_q$ of $T$ is an index $2$ subgroup of $\Aut(T)$.  
Thus, we let $\mathcal{I}=(\rho_1,\rho_2)$ 
where $\rho_1$ is the identity map and $\rho_2$ is an outer automorphism of $T$ such that $\rho_2^2=\rho_1$ 
(see Remark \ref{remaut2}), and we consider the code $C(T,\mathcal{I})$.
\begin{table}[h]
\centering
\begin{tabular}{c c c}
\hline\hline
Degree & $T$ & Conditions \\
\hline
6&$S_6, A_6$&--\\
11&$\PSL(2,11)$&--\\
12&$M_{12}$&--\\
15&$A_7$&--\\
176&$HS$&--\\
$(\ell^n-1)/(\ell-1)$&$\PSL(n,\ell)\leq T\leq \PGaL(n,\ell)$&$n>2$\\
\hline
\end{tabular}
\caption{$2$-transitive almost simple groups with two inequivalent actions.}\label{table1}
\end{table}

\begin{remark}\label{remaut2} For each $T$ in Table \ref{table1}, there exists an outer automorphism of $T$ 
  that is an involution.  For $T$ in the last line of Table \ref{table1}, the 
  automorphism induced by the inverse transpose map is the required outer automorphism.    
  For each of the other groups in Table \ref{table1}, we consult the character
  table of $T$ in the ATLAS \cite{at}. For each $T$ we find a conjugacy class of elements in
  $\Aut(T)\backslash\overline{N_{S_q}(T)}$ which has elements of order $2$.  
  (Here $\overline{N_{S_q}(T)}$ denotes the subgroup of $\Aut(T)$ induced by $N_{S_q}(T)$.) 
\end{remark}

\subsection{The Symmetric Group $T=S_6$}\label{secs6} By referring to the character table of $S_6$ in the ATLAS, we 
can determine the number of fixed points in each action for each conjugacy class of $S_6$.
By subtracting this from the degree of $S_6$, we determine, again for each action, the size of the support for
the elements in each conjugacy class of $S_6$. 
\begin{table}[h]
\centering
\begin{tabular}{l|c|c|c|c|c|c|c|c|c|c|c}
\hline\hline
Conjugacy Class&1A&2A&2B&2C&3A&3B&4A&4B&5AB&6A&6B\\
\hline
$|\supp(t\rho_1)|$&0&4&2&6&3&6&6&4&5&5&6\\
$|\supp(t\rho_2)|$&0&4&6&2&6&3&6&4&5&6&5\\\hline
Sum of supports&0&8&8&8&9&9&12&8&10&11&11\\\hline
\end{tabular}
\caption{The Symmetric Group $S_6$}
\label{s6}
\end{table} 
We give this information in Table \ref{s6}.
By summing the sizes of the supports for each conjugacy class, it follows from Proposition \ref{distlowerbound}
that $C(S_6,\mathcal{I})$ has minimum
distance $8$.  The minimal degree of $S_6$ is $2$ in both actions, and so $\Rep_2(C(S_6,\rho_i))$ has minimum distance $4$ for $i=1,2$.  
Thus, $C(S_6,\mathcal{I})$ has the same size and length as $\Rep_2(C(S_6,\rho_i))$ (for $i=1,2$), but has double the 
minimum distance.  In particular, $\delta(C(S_6,\mathcal{I}))$ is greater than the lower bound in Proposition \ref{distlowerbound}~(iii).  

\subsection{The Alternating Group $T=A_6$}\label{seca6}  Again, by referring to the ATLAS, we determine, for each action, the size of the
support for the elements in each conjugacy class of $A_6$.  We present this information in Table \ref{a6}.
It follows from Proposition \ref{distlowerbound} that we can read off the minimum distance of $C(A_6,\mathcal{I})$ from
Table \ref{a6}, which is $8$.  The minimal degree of
$A_6$ in both actions is $3$, so $\Rep_2(C(A_6,\rho_i))$ has minimum distance $6$ for $i=1,2$.  Thus $C(A_6,\mathcal{I})$
is strictly greater than the lower bound in Proposition \ref{distlowerbound} (iii).  
\begin{table}[h]
\centering
\begin{tabular}{l|c|c|c|c|c|c|c}
\hline\hline
Conjugacy Class&1A&2A&3A&3B&4A&5A&5B\\
\hline
$|\supp(t\rho_1)|$&0&4&3&6&6&5&5\\
$|\supp(t\rho_2)|$&0&4&6&3&6&5&5\\\hline
Sum of supports&0&8&9&9&12&10&10\\\hline
\end{tabular}
\caption{The Alternating Group $A_6$}
\label{a6}
\end{table} 

\subsection{The Mathieu Group $T=M_{12}$}\label{secm12}  In Table \ref{m12} we give the size of the support, for each action, of the elements 
in each conjugacy class of $M_{12}$ \cite{at}.  By Proposition \ref{distlowerbound}, we deduce
that $\delta(C(M_{12},\mathcal{I}))=16$. This is equal to the minimum distance of $\Rep_2(C(M_{12},\rho_i)$ for $i=1,2$ 
as the outer automorphism $\rho_2$ does not change the cycle structure of the elements in the conjugacy class $\mathrm{2B}$, for which 
the size of the support is equal to the minimal degree of $M_{12}$.
\begin{table}[h]
\centering
\begin{tabular}{l|c|c|c|c|c|c|c|c|c|c|c|c|c|c|c}
\hline\hline
Conjugacy Class&1A&2A&2B&3A&3B&4A&4B&5A&6A&6B&8A&8B&10A&11A&11B\\
\hline
$|\supp(t\rho_1)|$&0&12&8&9&12&12&8&10&12&11&12&10&12&11&11\\
$|\supp(t\rho_2)|$&0&12&8&9&12&8&12&10&12&11&10&12&12&11&11\\\hline
Sum of supports&0&24&16&18&24&20&20&20&24&22&22&22&24&22&22\\\hline
\end{tabular}
\caption{The Mathieu Group $M_{12}$}
\label{m12}
\end{table} 
However, the codes $C(M_{12},\mathcal{I})$ and
$\Rep_2(C(M_{12},\rho_i))$ for $i=1$~or~$2$ are inequivalent.  This can be seen by considering
the distance distribution of each code.  By Proposition \ref{distlowerbound}, 
both $C(M_{12},\mathcal{I})$ and $\Rep_2(C(M_{12},\rho_i))$ are distance
invariant.  Therefore, the respective inner distance distribution is equal to the distance distribution from any codeword. 
In $C(M_{12},\mathcal{I})$, the codewords that are at distance $16$ from $\alpha(1,\mathcal{I})$ are the codewords associated
with elements from the conjugacy class $\mathrm{2B}$.  Hence, in the inner distance distribution of $C(M_{12},\mathcal{I})$,
$a_{16}=495$, the size of the conjugacy class $\mathrm{2B}$ in $M_{12}$.  However, in $\Rep_2(C(M_{12},\rho_i))$, the codewords
that are at distance $16$ from $(\alpha(1,\rho_i),\alpha(1,\rho_i))$ are precisely the elements from the conjugacy classes $\mathrm{2B}$
and $\mathrm{4B}$ or $\mathrm{4A}$ respectively for $i=1$ or $2$.  Hence, in this case $a_{16}$ is equal
to the sum of the sizes of the conjugacy classes $\mathrm{2B}$ and $\mathrm{4B}$ or $\mathrm{4A}$ respectively, which is equal to $3465$ 
(note classes $\mathrm{4B}$ and $\mathrm{4A}$ contain the same number of elements).  
The inner distribution for each code can be calculated in this way and is given in Table \ref{distdist}.  Note, we only
give the non-zero terms of the inner distribution.
\begin{table}[h]
\centering
\begin{tabular}{l|c|c|c|c|c|c}
\hline\hline
$\kappa(C)$&$a_0$&$a_{16}$&$a_{18}$&$a_{20}$&$a_{22}$&$a_{24}$\\
\hline
$C(M_{12},\mathcal{I})$&1&495&1760&15444&56880&20460\\
$\Rep_2(C(M_{12},\rho_i))$&1&3465&1760&21384&33120&35310\\\hline
\end{tabular}
\caption{Inner distance distributions}
\label{distdist}
\end{table} 

In the remaining cases from Table \ref{table1}, we claim that $|\supp(t\rho_1)|=|\supp(t\rho_2)|$ for each $t\in T$.  Consequently
by Lemma \ref{diffreps}, $C(T,\mathcal{I})$ has the same inner distance distribution as $\Rep_2(C(T,\rho_i))$ for $i=1,2$,
and therefore the same minimum distance.  Hence, in these cases the code $C(T,\mathcal{I})$ has a minimum distance
that is equal to the lower bound in Proposition \ref{distlowerbound} (iii) even though $\rho_1\neq\rho_2$.    

\subsection{The Groups $T=\PSL(2,11),A_7, HS$}\label{secspor}  By referring to the ATLAS, we
see that in each case the permutation character of the action generated by $\rho_1$ is the same as 
the permutation character of the action generated by $\rho_2$.  
In particular, for $T=\PSL(2,11)$, $A_7$ or $HS$, the permutation character for both actions 
is equal to $\mathrm{1A}+\mathrm{10B}$, $\mathrm{1A}+\mathrm{14B}$, or $\mathrm{1A}+\mathrm{175A}$ respectively \cite[p.7,10,80]{at}.  
Hence $|\fix(t\rho_1)|=|\fix(t\rho_2)|$ for all $t\in T$ and so $|\supp(t\rho_1)|=|\supp(t\rho_2)|$ for all $t\in T$.

\subsection{The Projective Linear Groups $\PSL(n,\ell)\leq T\leq \PGaL(n,\ell)$}\label{secpl}  In its natural action $T$ is 
acting $2$-transitively on $\mathcal{P}$, the set of $(\ell^n-1)/(\ell-1)$ one-dimensional subspaces of $V=\F_\ell^n$.  
Moreover, under $\rho_2$, the action of $T$ is permutationally isomorphic to the action of 
$T$ on $\mathcal{B}$, the set of $(\ell^n-1)/(\ell-1)$ hyperplanes of $V$.  It 
holds that each hyperplane is uniquely determined by the set of one-dimensional subspaces it contains.  Consequently 
$\mathcal{D}=(\mathcal{P},\mathcal{B})$ forms a symmetric $2$-design, and in particular, $T$ is a
group of automorphisms of $\mathcal{D}$.  Consequently, the cycle structure of $t\rho_1$ is 
the same as that for $t\rho_2$, for all $t\in T$, see \cite[Cor. 3.2]{lander}.  
Hence, $|\supp(t\rho_1)|=|\supp(t\rho_2)|$ for all $t\in T$.

An open question for the cases in Sections \ref{secspor} and \ref{secpl} is whether $C(T,\mathcal{I})$
is equivalent to $\Rep_r(C(T,\rho_i))$, for $i=1$ or $2$, under the automorphisms of the Hamming graph.  

\subsection{Neighbour transitivity}\label{secneitran}  Let $C$ be a code in $H(m,q)$.  For any vertex $\nu$ in $H(m,q)$ we 
let $d(\nu,C)=\min\{d(\nu,\beta)\,:\,\beta\in C\}$ and $C_i=\{\nu\,:\,d(\nu,C)=i\}$.  
We call $C_1$ the \emph{set of neighbours of $C$}, and if there exists 
a group of automorphisms $G$ such that both $C$ and $C_1$ are $G$-orbits then we say that $C$ is \emph{$G$-neighbour
transitive}, or simply \emph{neighbour transitive}.  

Throughout this section $T$ is one of the groups from Table \ref{table1}, and
$C(T,\mathcal{I})$ is the code generated by $\mathcal{I}=(\rho_1,\rho_2)$ where $\rho_1$ is
the identity map and $\rho_2$ is an outer automorphism of $T$ such that $\rho_2^2=\rho_1$.  
As we mentioned in the introduction, one of the motivations for considering 
twisted permutation codes comes from the family of neighbour transitive permutation codes classified in 
\cite{diagnt} and their subsequent neighbour transitive repetition constructions.  In this section we prove
the following for the codes presented above.

\begin{theorem}\label{ntthm} For each $T$ in Table \ref{table1}, the code $C(T,\mathcal{I})$ is neighbour transitive.    
\end{theorem}

Before we prove Theorem \ref{ntthm}, we first show that
any automorphism of $T$ defines an automorphism of $N_{S_q}(T)$.  
We define the following homomorphism:     
\begin{equation}\label{tashom}\begin{array}{c c c c}  
     \vartheta:&N_{S_q}(T)&\longrightarrow&\Aut(T)\\ 
&y&\longmapsto&\overline{y}
\end{array}\end{equation} where $t^{\overline{y}}=y^{-1}ty$ for all $t\in T$, and we denote the
image of $N_{S_q}(T)$ by $\overline{N_{S_q}(T)}$.   
Since $T$ is acting $2$-transitively, and therefore primitively, 
it follows that $\ker(\vartheta)=C_{S_q}(T)=1$ \cite[Theorem 4.2A]{dixmort}.  Hence $N_{S_q}(T)\cong\overline{N_{S_q}(T)}$. 
For each group in Table \ref{table1} it is known that $\overline{N_{S_q}(T)}$ 
is a subgroup of index $2$ in $\Aut(T)$, and therefore is normal in $\Aut(T)$.  
Thus for $\rho\in\Aut(T)$ and $y\in N_{S_q}(T)$, $\rho^{-1}\overline{y}\rho\in\overline{N_{S_q}(T)}$.
Moreover, because $N_{S_q}(T)\cong \overline{N_{S_q}(T)}$, there exists a unique $y'\in N_{S_q}(T)$ such 
that $\rho^{-1}\overline{y}\rho=\overline{y'}$.  Thus, for $\rho\in\Aut(T)$,  $\hat{\rho}:N_{S_q}(T)\longmapsto N_{S_q}(T)$ given by 
\begin{equation}\label{normact} y\hat{\rho}=\vartheta^{-1}(\rho^{-1}\overline{y}\rho)=y'\end{equation} is a well defined
automorphism of $N_{S_q}(T)$.  We note that \eqref{normact} implies $\overline{y\hat{\rho}}=\rho^{-1}\overline{y}\rho$.  To simplify the notation,
for $y\in N_{S_q}(T)$ we write $y\rho$ for $y\hat{\rho}$, and regard $\rho$ as a representation of $N_{S_q}(T)$.  

The code $C(T,\mathcal{I})$ is contained in the vertex set of $\Gamma^{2}=H(2q,q)$, 
which we can identify with the set of arbitrary $2$-tuples of vertices from $\Gamma=H(q,q)$.
Thus, given arbitrary automorphisms $x,y\in\Aut(\Gamma)$, we let $(x,y)\in\Aut(\Gamma)\times\Aut(\Gamma)$ 
act on the vertices of $\Gamma^{2}$ in the following way: 
\begin{equation}\label{2act}(\alpha_1,\alpha_2)^{(x,y)}=(\alpha^x_1,\alpha^y_2),\end{equation} where $\alpha_1,\alpha_2\in V(\Gamma)$. 
We now construct a group of automorphisms of $C(T,\mathcal{I})$ that stabilises $\alpha(1,\mathcal{I})$.  
To do this we first
construct an automorphism of $C(T,\rho)$, for any $\rho\in\Aut(T)$, in $\Gamma=H(q,q)$. Now $\rho$
defines an automorphism of $N_{S_q}(T)$, so $y\rho\in N_{S_q}(T)$ for $y\in N_{S_q}(T)$, and we let
$x_{y\rho}=(y\rho,\ldots,y\rho)\in B\cong S_q^q$, $\sigma(y\rho)\in L\cong S_q$ and $a(y,\rho)=x_{y\rho}\sigma(y\rho)\in\Aut(\Gamma)$.
Suppose that $i^{y\rho}=j$.  Then it follows that 
$$\alpha(t,\rho)^{a(y,\rho)}|_j=i^{t\rho y\rho}=j^{(y\rho)^{-1}t\rho y\rho}=j^{(y^{-1}ty)\rho},$$ and hence
\begin{equation}\label{codeact}\alpha(t,\rho)^{a(y,\rho)}=\alpha(y^{-1}ty,\rho).\end{equation}  
We now define $$A(T,\mathcal{I})=\{a(y,\mathcal{I})=(a(y,\rho_1),a(y,\rho_2))\,:\,y\in N_{S_q}(T)\},$$ where
we regard $\rho_1$, $\rho_2$ as representations of $N_{S_q}(T)$.  
Allowing $A(T,\mathcal{I})$ to act on vertices of $\Gamma^2$ as in \eqref{2act}, 
it follows from \eqref{codeact} that $$\alpha(t,\mathcal{I})^{a(y,\mathcal{I})}=\alpha(y^{-1}ty,\mathcal{I})$$
for all $t\in T$.  As $y\in N_{S_q}(T)$ we deduce that
$A(T,\mathcal{I})$ is a group of automorphisms $C(T,\mathcal{I})$ that stabilises 
$\alpha(1,\mathcal{I})$.

Let $\sigma$ be the automorphism of $\Gamma^{2}$ that maps $(\alpha_1,\alpha_2)$
to $(\alpha_2,\alpha_1)$ for all $\alpha_1,\alpha_2\in V(\Gamma)$.  We observe 
that $\alpha(t,\rho)=\alpha(t\rho,\rho_1)$ for any $\rho\in\Aut(T)$ and
any $t\in T$ (recall $\rho_1$ is the identity automorphism).
Hence, recalling that $\rho_2^2=\rho_1$, it follows that
$$\alpha(t,\mathcal{I})^\sigma=(\alpha(t,\rho_2),\alpha(t,\rho_1))=(\alpha(t\rho_2,\rho_1),\alpha(t\rho_2,\rho_2))=\alpha(t\rho_2,\mathcal{I}).$$
Thus $\sigma$ is also an automorphism of $C(T,\mathcal{I})$ that stabilises $\alpha(1,\mathcal{I})$.  

\begin{lemma}\label{ntlem} Let $H=\langle A(T,\mathcal{I}),\sigma \rangle$.  Then $H$ acts transitively
on $\Gamma^2_1(\alpha(1,\mathcal{I}))$.  
\end{lemma}  

\begin{proof} We first describe the 
neighbours of the codeword $\alpha(1,\rho)$ for $\rho\in\Aut(T)$
in $\Gamma=H(q,q)$.  Following the notation of \cite{diagnt}, for $1\leq i,j,k\leq q$ we let 
\[\nu(\alpha(1,\rho),i,j)|_k=\left\{\begin{array}{ll}
  k&\textnormal{if $k\neq i,$}\\ j &\textnormal{if 
   $k=i$,} \end{array}\right.\]
so $\Gamma_1(\alpha(1,\rho))=\{\nu(\alpha(1,\rho),i,j)\,:\,i\neq j\}$.  It follows from 
\cite[Lemma 1]{diagnt} and \eqref{codeact} that $$\nu(\alpha(1,\rho),i,j)^{a(y,\rho)}=\nu(\alpha(1,\rho)^{a(y,\rho)},i^{y\rho},j^{y\rho})
=\nu(\alpha(1,\rho),i^{y\rho},j^{y\rho}).$$
Thus, because $N_{S_q}(T)$ is acting $2$-transitively, we deduce that $A(T,\rho)=\{a(y,\rho)\,:\,y\in N_{S_q}(T)\}$ 
acts transitively on $\Gamma_1(\alpha(1,\rho))$ in $\Gamma$.  Hence 
$A(T,\mathcal{I})$ has two orbits on $\Gamma^2_1(\alpha(1,\mathcal{I}))$ in $\Gamma^2$, which are
$$\mathcal{O}_1=\{(\nu,\alpha(1,\rho_2))\,:\,\nu\in\Gamma_1(\alpha(1,\rho_1))\}\textnormal{ and }
\mathcal{O}_2=\{(\alpha(1,\rho_1),\nu)\,:\,\nu\in\Gamma_1(\alpha(1,\rho_2))\}.$$
Because $1\rho_2=1$, it follows that $\alpha(1,\rho_2)=\alpha(1\rho_2,\rho_1)=\alpha(1,\rho_1)$, and so 
$\Gamma_1(\alpha(1,\rho_1))=\Gamma_1(\alpha(1,\rho_2))$.  Thus $\sigma$ interchanges $\mathcal{O}_1$ and $\mathcal{O}_2$. 
\end{proof}

We prove Theorem \ref{ntthm} by considering the group $G=\langle\Diag(T,\mathcal{I}), A(T,\mathcal{I}),\sigma\rangle$.
It follows from Lemma~\ref{distlowerbound} that $\Diag(T,\mathcal{I})$ acts regularly on $C(T,\mathcal{I})$.  Adding to this
the result of Lemma~\ref{ntlem}, we deduce that $G$ is a group of automorphisms of $C(T,\mathcal{I})$ that
acts transitively on $C(T,\mathcal{I})$.  
It also follows from Lemma~\ref{ntlem} that the stabiliser of $\alpha(1,\mathcal{I})$ in $G$ is equal to $H$.  
By allowing $G$ to translate any two neighbours of $C(T,\mathcal{I})$
to two neighbours of $\alpha(1,\mathcal{I})$, and then allowing an element of $H$ to map one of these neighbours to the other, we
deduce that $G$ acts transitively on the set of neighbours of $C(T,\mathcal{I})$, which proves Theorem \ref{ntthm}.  

\section{The affine special linear group $\ASL(2,r)$}\label{secaff}

Let $V=\mathbb{F}_r^2$ be the $2$-dimensional 
vector space of row vectors over the finite field $\mathbb{F}_r$ of size $r=2^f$ for some positive integer $f\geq 2$.  The group 
$\ASL(2,r)$ is equal to the split extension of $N$, the translations of $V$, by $\SL(2,r)$, the group of invertible
$2\times 2$ matrices over $\F_r$ with determinant $1$, and $\ASL(2,r)$ has a natural $2$-transitive
action on $V$.  It is known that there are $r$ conjugacy classes of complements of $N$ in $\ASL(2,r)$ \cite{boh}.
By embedding $\ASL(2,r)$ into $\SL(3,r)$, in this section we construct a representative for each of these conjugacy classes.  
Then, by considering the coset action on each representative, we give an explicit construction of the $r$ 
inequivalent $2$-transitive actions of $\ASL(2,r)$ of degree $r^2$.  We let $\mathcal{I}=(\rho_1,\ldots,\rho_{r})$,
where the $\rho_i$ are the representations for these inequivalent actions, and we prove the following.

\begin{theorem}\label{affmain} Let $\mathcal{I}$ be as above.  Then $C(\ASL(2,r),\mathcal{I})$ has the same inner distance 
distribution, and therefore also the same minimum distance, as $\Rep_r(C(\ASL(2,r),\rho_i))$ for $i=1,\ldots,r$.
\end{theorem}

\begin{remark} The code in Theorem \ref{affmain} shows us that given any even prime power $r$, we can construct
a twisted permutation code with $r$ pairwise distinct representations such that the
code has a minimum distance equal to the lower bound of Proposition \ref{distlowerbound} (iii).
\end{remark}

To embed $\ASL(2,r)$ into $\SL(3,r)$ we begin by taking an element $a$ 
of order $r-1$ in $\mathbb{F}_r^{*}$ and letting $b=(1+a^2)^{-1}$. As $f\geq 2$, we see that $b$ is well-defined.
Consider the matrices 
\begin{equation}\label{gpelts}
x=\left(
\begin{array}{cc}
a&0\\
0&a^{-1}\\
\end{array}
\right),\,\,
y=\left(
\begin{array}{cc}
1&1\\
0&1\\
\end{array}
\right),\,\,
z=\left(
\begin{array}{cc}
b&b^2+b+1\\
1&b+1\\
\end{array}
\right).\end{equation}  For every $w\in \mathbb{F}_r$ we define 
\begin{equation*}
u=wa+wb+w\quad \textrm{and}\quad v=w+wa,
\end{equation*} 
\[
X=\left(
\begin{array}{ccc}
1&0&0\\
0&a&0\\
0&0&a^{-1}\\
\end{array}
\right),\,\,
Y_w=\left(
\begin{array}{ccc}
1&0&v\\
0&1&1\\
0&0&1\\
\end{array}
\right),\,\,
Z_w=\left(
\begin{array}{ccc}
1&w&u\\
0&b&b^2+b+1\\
0&1&b+1\\
\end{array}
\right)
\] and $$S_w=\langle X,Y_w,Z_w\rangle.$$ 

\begin{lemma}\label{apeman1}We have  $\SL(2,r)=\langle x,y,z\rangle$ and $\SL(2,r)\cong S_w$.
\end{lemma}
\begin{proof}
As the field element $a$ has order $r-1$, we have
\begin{equation}\label{eqx}
x^{r-1}=1\quad\textrm{and}\quad X^{r-1}=1.
\end{equation}Also, a direct computation shows that
\begin{equation}\label{eqy}
y^2=1\quad\textrm{and}\quad Y_w^2=1.
\end{equation}
The characteristic polynomial of $z$ is $(\Lambda-b)(\Lambda-(b+1))+b^2+b+1=\Lambda^2+\Lambda+1$ and hence $z^2+z+1=0$.  
From this it follows with an easy computation that
\begin{eqnarray}\label{eqz}
z^3=1\quad \textrm{and}\quad Z_w^3=1.
\end{eqnarray}
Using the definition of $a$, $b$, $X$, $Y_w$ and $Z_w$, we see with a rather long  (but simple and direct) computation that
\begin{equation}\label{eqmany}
(xz)^2=(yz)^2=[x^i,y]^2=1\,\,\textrm{and}\,\, (XZ_w)^2=(Y_wZ_w)^2=[X^i,Y_w]^2=1,
\end{equation}
for every $i\in \{0,\ldots,r-1\}$.

Since $r$ is even, we see that $a^2$ is also a generator of the multiplicative group of the field $\mathbb{F}_r$. 
So, let $p(\Lambda)=c_f\Lambda^f+c_{f-1}\Lambda^{f-1}+c_{f-2}\Lambda^{f-2}+\cdots +c_{1}\Lambda+c_0$ be the minimal polynomial of $a^2$ over the 
ground field $\mathbb{F}_2$, that is, $c_i\in \{0,1\}$ and $p(a^2)=0$. Now a computation shows that
\begin{eqnarray*}
y^{c_0}xy^{c_1}xy^{c_2}x\cdots xy^{c_f}&=&
\left(
\begin{array}{cc}
a^f&\zeta\\
0&a^{-f}
\end{array}
\right)
\end{eqnarray*}
with $$\zeta=c_fa^f+c_{f-1}a^{f-2}+c_{f-2}a^{f-4}+\cdots+c_1a^{-f+2}+c_0a^{-f}.$$
In particular, $a^f\zeta=p(a^2)=0$ and hence $\zeta=0$. This gives that
\begin{eqnarray}\label{hard}
x^f&=&y^{c_0}xy^{c_1}xy^{c_2}x\cdots xy^{c_f}.
\end{eqnarray}
Similarly, we have
\begin{eqnarray}\label{hard1}
Y_w^{c_0}XY_w^{c_1}XY_w^{c_2}X\cdots XY_w^{c_f}&=&
\left(
\begin{array}{ccc}
1&0&v\zeta\\
0&a^f&\zeta\\
0&0&a^{-f}
\end{array}
\right)=X^f.
\end{eqnarray}

As $x$, $y$ and $z$ have determinant $1$, we have $\langle x,y,z\rangle\leq \SL(2,r)$.  Also, we 
see from~\cite[Section~$7.6$, lines~9,~10]{GKK} that $\SL(2,r)$ has presentation with three generators $R$, $S$ and $U$ and with relators
\begin{eqnarray*}
&& R^{r-1}=S^2=U^3=(RU)^2=(SU)^2=1,\\
&&[R^i,S]^2=1\,\,\,\,\,\textrm{for every  }i\in \{1,\ldots,r-1\},\\
&&R^f=S^{c_0}RS^{c_1}RS^{c_2}R\cdots S^{c_{f-1}}RS^{c_f}.
\end{eqnarray*}
In particular, from~\eqref{eqx}--\eqref{hard1} we obtain 
that $\langle x,y,z\rangle$ and $S_w$  both satisfy the defining relations of $\SL(2,r)$ and hence are isomorphic to a 
quotient of $\SL(2,r)$. Since $f\geq 2$, the group $\SL(2,r)$ is simple and the lemma follows.
\end{proof}

For $\n\in V$ let $\varphi_{\n}$ denote the translation of $V$ by $\n$.  By letting
$\varphi_{\n}$, $x$, $y$ and $z$ act on an arbitrary vector of $V$, we determine the following relations.
\begin{equation}\label{neq1}
x^{-1}\varphi_{\n}x\varphi_{{\bf{\n}}x}=y^{-1}\varphi_{\n}y\varphi_{{\bf{v}}y}=z^{-1}\varphi_{\n}z\varphi_{{\bf{v}}z}=1\textnormal{ for all $\n\in V$.}
\end{equation}
If $\n=(v_1,v_2)\in V$ we let  
\[
e(\n)=\left(
\begin{array}{ccc}
1&v_1&v_2\\
0&1&0\\
0&0&1
\end{array}
\right)
\] and $$E=\{e(\n)\,:\,\n\in V\}.$$  
Direct calculation shows that, for $\n\in V$ and $w\in \F_r$, the following relations hold:
\begin{equation}\label{neq2}
X^{-1}e(\n)Xe(\n x)=Y_w^{-1}e(\n)Y_we(\n y)=Z^{-1}_we(\n)Z_we(\n z)=1.\end{equation}

It is a consequence of \eqref{neq2} that $S_w$ normalises $E$ for each $w$.  Furthermore, by Lemma \ref{apeman1},
$S_w\cong\SL(2,r)$, which is simple as $f\geq 2$, and so $E\cap S_w=1$ for each $w$.  Thus we can define
$$G=ES_0,$$ the split extension of $E$ by $S_0$.  As $Y_w=e({\bf{v}}_1)Y_0$ and $Z_w=e({\bf{v}}_2)Z_0$ where  ${\bf{v}_1}=(0,v)$, 
${\bf{v}_2}=(w,u)\in V$, we conclude that $S_w$ is a subgroup of $G$ for each $w$.  
Hence $G$ is also equal to the split extension of $E$ by $S_w$ for each $w$.  Thus another consequence of Lemma \ref{apeman1} is
that $$\ASL(2,r)=\langle \varphi_{\n},x,y,z\,:\,\n\in V\rangle,\textnormal{ and } G=\langle e(\n), X, Y_w, Z_w\,:\,\n\in V\rangle\textnormal{ for each
$w\in\F_r$.}$$ For $w\in\F_r$ let $\tau_w:\ASL(2,r)\longrightarrow G$ be the group homomorphism that 
takes $$x\longmapsto X,\,\, y\longmapsto Y_w,\,\, z\longmapsto Z_w,\,\,\textnormal{ and }\,\,\varphi_{\n}\longmapsto e(\n)
\textnormal{ for each $\n\in V$.}$$  We observe, from \eqref{eqx}--\eqref{neq2}, that $\tau_w$ is well defined.
Furthermore, because $E$ and $S_w$ are both subgroups of $\tau_w(\ASL(2,r))$, we deduce the following.

\begin{lemma} The map $\tau_w$ is an isomorphism from $\ASL(2,r)$ to $G$ for each $w\in\F_r$.  
\end{lemma}  

Note that since $H^1(\SL(2,r),N)\cong \mathbb{F}_r$ by~\cite[Table~$4.3$, type~$A_1$]{boh}, 
we have that $\ASL(2,r)$ contains exactly $r$ conjugacy classes of complements of $N$ in $\ASL(2,r)$.  We now
show that in this embedding of $\ASL(2,r)$ in $\SL(3,r)$, each conjugacy class of complements of $N$ 
contains a unique group $S_w$ for some $w\in \F_r$.     

\begin{lemma}\label{apeman2}For each $w$ and $w'$ in $\mathbb{F}_r$ with $w\neq w'$, the groups $S_w$ and $S_{w'}$ are not conjugate in $G$.
\end{lemma}
\begin{proof}
Let $w$ and $w'$ be in $\mathbb{F}_r$ and suppose that $S_w$ and $S_{w'}$ are conjugate in $G$.
As $G=ES_w=ES_{w'}$ and $E\cap S_w=E\cap S_{w'}=1$, it follows that 
$S_w$ and $S_{w'}$ are conjugate via an element $e(\n)\in E$ for some $\n=(v_1,v_2)\in V$. 
Furthermore, as $X\in S_w\cap S_{w'}$, we have $X^{-1},X^{e(\n)}\in S_{w'}$ and hence $X^{-1}X^{e(\n)}\in S_{w'}$. Now
\begin{eqnarray*}
X^{-1}X^{e(\n)}&=&
\left(
\begin{array}{ccc}
1&0&0\\
0&a^{-1}&0\\
0&0&a\\
\end{array}\right)
\left(
\begin{array}{ccc}
1&v_1&v_2\\
0&1&0\\
0&0&1\\
\end{array}\right)
\left(
\begin{array}{ccc}
1&0&0\\
0&a&0\\
0&0&a^{-1}\\
\end{array}\right)
\left(
\begin{array}{ccc}
1&v_1&v_2\\
0&1&0\\
0&0&1\\
\end{array}\right)\\
&=&
\left(
\begin{array}{ccc}
1&v_1+v_1a&v_2+v_2a\\
0&1&0\\
0&0&1\\
\end{array}\right).
\end{eqnarray*}
Since $E\cap S_{w'}=1$ we must have $v_1+v_1a=0$ and $v_2+v_2a=0$, that is, $v_1a=v_1$ and $v_2a=v_2$. As $f\geq 2$, 
we have $a\neq 1$ and hence $v_1=v_2=0$. This gives $e(\n)=1$ and hence $S_w=S_{w'}$.

From the previous paragraph we have $Y_w,Y_{w'}\in S_{w}$ and hence $Y_{w}Y_{w'}\in S_w$. Now,
\begin{eqnarray*}
Y_wY_{w'}&=&\left(
\begin{array}{ccc}
1&0&w+wa\\
0&1&1\\
0&0&1\\
\end{array}\right)\left(
\begin{array}{ccc}
1&0&w'+w'a\\
0&1&1\\
0&0&1\\
\end{array}\right)\\
&=&\left(
\begin{array}{ccc}
1&0&(w+w')+(w+w')a\\
0&1&0\\
0&0&1\\
\end{array}\right).
\end{eqnarray*}
Since $E\cap S_w=1$, we must have $(w+w')a=w+w'$ and hence $w+w'=0$. This gives $w=w'$ and the lemma follows.
\end{proof}

Let $\Omega_w$ be the set of right cosets of $S_w$ in $G$.  As $S_w$ normalises $E$, and because $G=ES_w$, it
follows that every right coset of $S_w$ in $\Omega_w$ has a coset representative in $E$. Moreover, because $E\cap S_w=1$, 
each $e(\n)\in E$ uniquely determines the right coset it belongs to.  Hence the map 
$\iota_w: V\longrightarrow \Omega_w$ given by $\n \longmapsto S_we(\n)$ is a well defined bijection.  

\begin{proposition}\label{permiso} Under $(\iota_w,\tau_w)$, the action of $G$ on $\Omega_w$ is permutationally
isomorphic to the action of $\ASL(2,r)$ on $V$.
\end{proposition}

\begin{proof}  Let $\n\in V$ and consider the generator $y$ of $\ASL(2,r)$.  From the definition of $\iota_w$
it follows that $\iota_w({\n}y)=S_we({\bf{v}}y)$.  Now, it follows from \eqref{neq1} that 
$$\iota_w(\n)\tau_w(y)=S_we(\n)Y_w=S_wY_we({\n}y)=S_we({\n}y).$$  Thus
$\iota_w({\n}y)=\iota_w(\n)\tau_w(y)$.  By applying similar arguments to the other generators of $\ASL(2,r)$, 
we conclude that $(\iota_w,\tau_w)$ is a permutational isomorphism.  
\end{proof}

It is a consequence of Proposition \ref{permiso} that $G$ acts $2$-transitively on $\Omega_w$ for each $w\in\F_r$.  
Moreover, it follows from Lemma \ref{apeman2} and \cite[Thm. 1.3]{pcam} that for $w'\neq w$ the action of $G$ on $\Omega_{w'}$ is inequivalent
to the action of $G$ on $\Omega_w$. Hence these actions of $G$ on $\Omega_w$, for $w$ in $\F_r$, are $r$ pairwise inequivalent
$2$-transitive actions of degree $r^2$. Thus, by considering the action of $\ASL(2,r)$ on the right cosets of $\tau_0^{-1}(S_w)$ for each $w\in S_w$,
it follows that $\ASL(2,r)$ has $r$ inequivalent $2$-transitive actions of degree $r^2$.  To prove Theorem \ref{affmain},
we look at the number of fixed points of an element of $\ASL(2,r)$ in each of these actions.  

\begin{lemma}\label{atleastr} Let $1\neq t\in\ASL(2,r)$ be an element that has at least two fixed points 
in some $2$-transitive representation of degree $r^2$.  Then $t$ fixes $r$ points in every $2$-transitive representation.  Moreover $|t|=2$.  
\end{lemma}

\begin{proof}  By Proposition \ref{permiso}, each $2$-transitive action of $\ASL(2,r)$ of degree $r^2$ is permutationally isomorphic
to its natural action on $V$.  Thus without loss of generality, we can assume
that $t$ fixes at least $2$ points in the natural action of $\ASL(2,r)$.  As $\ASL(2,r)$ acts
$2$-transitively on $V$, it follows that $t$ is conjugate to an element that fixes ${\bf{0}}=(0,0)$ and ${\bf{e_2}}=(0,1)$.  
It holds that \begin{equation}\label{slstab}\SL(2,r)_{{\bf{e_2}}}=\left\{M_c=\left(
\begin{array}{cc}
1&c\\
0&1\\
\end{array}
\right)\,:\,c\in\F_r\right\},\end{equation}
so $t$ is conjugate to $M_{c'}$ for some $c'\in\F^*_r$.  
It is straightforward to show that $\fix(M_{c'})=\{(0,c)\,:\,c\in \F_r\}$ and $|M_{c'}|=2$.  So
$t$ fixes $r$ points in this action and has order $2$. 
Now, as $r$ is an even prime power, $c'=\ell^2$ for some $\ell\in\F_r$.  
Conjugating $M_{c'}$ by the diagonal matrix $\diag(\ell,\ell^{-1})$ gives the group element 
$y$ in \eqref{gpelts}.  Moreover, for
$0\neq w$, it holds that $\tau^{-1}_0\tau_w(y)=\tau_0^{-1}(Y_w)=\varphi_{{\bf{v}}}y$, where ${\bf{v}}=(0,v)$ with $v=w+wa$, and 
$\fix(\varphi_{{\bf{v}}}y)=\{(v,c)\,:\,c\in\F_r\}$.  Hence $\tau_0^{-1}\tau_w(t)$ fixes $r$ points.  In particular, $t$ fixes $r$ points in every 
$2$-transitive representation.  
\end{proof}

\begin{corollary}\label{corfix} Let $t\in\ASL(2,r)$.  Then $t$ fixes $0$, $1$ or $r$ points in each $2$-transitive representation of $\ASL(2,r)$ of degree $r^2$.  
\end{corollary}

\begin{lemma}\label{lem1} Let $t\in\ASL(2,r)$ be a non-trivial element of odd order.  Then $t$ fixes exactly one point in each $2$-transitive representation
of $\ASL(2,r)$.  
\end{lemma}

\begin{proof} Suppose in some $2$-transitive representation of $\ASL(2,r)$ that $t$ has no fixed points.  
Then in that representation $t$ has a cycle of minimal length $b>1$.  As the order of $t$ is equal
to the least common multiple of the disjoint cycle lengths of $t$, it follows that $b$ is odd as $|t|$ is odd.  
If all the disjoint cycles have length $b$ then $b$ divides $r^2$ (because $t$ has no fixed points), which is a contradiction
as $r=2^f$. Thus there exists a cycle of length $c>b$.  
Hence $t^b\neq 1$, and $t^b$ fixes at least $b>1$ points.  By Lemma \ref{atleastr}, it follows that $t^b$ fixes $r$ points
and $|t^b|=2$.  Hence $t^{2b}=1$.  Now, because $t^b\neq 1$ and $|t|$ is a multiple of $b$, it follows that $|t|=2b$, contradicting the fact
that $|t|$ is odd.  Thus $t$ fixes at least one point in every representation.  If $t$
fixes more than one point in some representation, then by Lemma \ref{atleastr}, $|t|=2$, which is a contradiction.
\end{proof}

\begin{lemma}\label{lem2}  Let $t\in\ASL(2,r)$ be an element of even order.  Then $t$ fixes the same number of points in each $2$-transitive representation 
of $\ASL(2,r)$.  
\end{lemma}

\begin{proof}  Suppose $|t|=2$.  Then, in any $2$-transitive representation of $\ASL(2,r)$, 
$t$ can be written as the disjoint union of transpositions.  Thus $t$ cannot fix exactly one point, otherwise 
$2$ would divide $r^2-1$, which is a contradiction.  Hence, by Corollary \ref{corfix}, $t$ either fixes $r$ points or $0$ points.  By Lemma \ref{atleastr}, 
if $t$ fixes $r$ points, then $t$ fixes $r$ points in each $2$-transitive representation of $\ASL(2,r)$.  
Hence if $t$ has no fixed points, then $t$ has no fixed points in every $2$-transitive representation of $\ASL(2,r)$.  

Now assume that $|t|>2$.  Because $|t|>2$, it follows from
Lemma \ref{atleastr} that $t$ cannot fix $r$ points. So, in every $2$-transitive representation, $t$ fixes either $0$ or $1$ point.
Suppose $t$ fixes $1$ point in some $2$-transitive representation, and let $b$ be the minimal non-trivial cycle length of $t$ in that representation.  
Further suppose that $t^b=1$, so $|t|=b$.  Then the disjoint union of the points in the $b$-cycles of $t$ is equal to
support of $t$, which has size $r^2-1$.  Hence $b$ divides $r^2-1$.
Thus $b$ is odd, contradicting the fact that $|t|$ is even.  Hence $t^b\neq 1$.  Now, $t^b$ fixes at least $1+b$ points, so by
Lemma \ref{atleastr}, $t^b$ fixes $r$ points and $|t^b|=2$.  Hence $|t|=2b$.  Moreover, the set of $r-1$ fixed points of $t^b$
points that are not fixed by $t$ is equal the disjoint union of the points that form the $b$-cycles of $t$.  Hence $b$ divides $r-1$, and so $b$ is odd.
However, as $t$ fixes one point, it follows that $t$ is conjugate to some $h\in\SL(2,r)$, and so $t^b$ is conjugate to $h^b\in\SL(2,r)$.
Now, $\SL(2,r)$ has only one conjugacy class of involutions (see for example \cite[p. 95]{grove}), 
so $h^b$ is conjugate the element $y$ in \eqref{gpelts}.  It is straightforward to show that the subgroup $\SL(2,r)_{{\bf{e_2}}}$ from 
(\ref{slstab}) is the centraliser in $\SL(2,r)$ of $y$, and has order $r$.  
Hence, because $h$ centralises $h^b$ it follows that $|h|$ divides $r$, and so $|t|$ divides $r$, contradicting the fact that $|t|=2b$ with $b$ odd.  
So $t$ has no fixed points in each $2$-transitive representation of $\ASL(2,r)$.  
\end{proof}

To prove Theorem \ref{affmain}, we let 
$\mathcal{I}=(\rho_1,\ldots,\rho_{r})$, where the $\rho_i$ are the representations of the $r$ inequivalent $2$-transitive actions of $\ASL(2,r)$ of
degree $r^2$.  It is a consequence of Lemmas \ref{lem1} and \ref{lem2} that $|\fix(t\rho_i)|=|\fix(t\rho_j)|$ for
all $i,j$ and for all $t\in\ASL(2,r)$.  Hence $|\supp(t\rho_i)|=|\supp(t\rho_j)|$ for all $i,j$ and for all $t\in\ASL(2,r)$.
Thus Theorem \ref{affmain} now follows from Lemma \ref{diffreps}.  

\section{Coset Actions on the Computer}\label{seccomp}

In this final section we use the algebraic computer software GAP \cite{GAP4} to analyse two further groups acting on sets of 
cosets of certain subgroups.

\subsection{The Affine Group $\ASL(3,2)$.}\label{secasl} We first consider the $2$-transitive action of the affine group 
$\AGL(3,2)$ on $\F_2^3$.  It is well known that the number of distinct actions of a $2$-transitive group 
$G$ of affine type is equal to the order of the cohomology group $H^1(G_0,N)$ \cite[Sec. 7.3]{pcam}, 
where $G_0$ is the stabiliser of a point in the natural action of $G$ and $N$ is the unique minimal normal subgroup of $G$.  
This in turn is equal to the number of $G$-conjugacy classes of the complements of $N$ in $G$.  For the group $\AGL(3,2)$, 
there are two distinct actions \cite[Table 7.3]{pcam}.  To generate these actions, 
we use the ``Complementclasses" function in GAP to find two representatives $H_1, H_2$ 
of these conjugacy classes.  
\begin{table}[h]
\centering
\begin{tabular}{l|c|c|c|c|c|c|c|c|c|c|c}
\hline\hline
Conjugacy Class&1&2&3&4&5&6&7&8&9&10&11\\
\hline
$|\supp(t\rho_1)|$&0&8&4&8&8&6&8&6&8&7&7\\
$|\supp(t\rho_2)|$&0&8&8&4&8&8&6&6&8&7&7\\\hline
Sum of supports&0&16&12&12&16&14&14&12&16&14&14\\\hline
\end{tabular}
\caption{$\ASL(3,2)$}
\label{asl32}
\end{table} If $\Omega_i$ is the set of right cosets of $H_i$ in $\ASL(3,2)$, 
we can construct in GAP the representation $\rho_i:\ASL(3,2)\longrightarrow \Sym(\Omega_i)$
that describes the action of $\ASL(3,2)$ on $\Omega_i$, for $i=1,2$.  
The group $\AGL(3,2)$ has $11$ conjugacy classes $\mathcal{C}_j$, and for each one, we 
calculate $|\supp(t\rho_i)|$ for some $t\in\mathcal{C}_j$ for $i=1,2$ and $j=1,\ldots,11$.  We give this information in Table \ref{asl32}.
The minimal degree of both actions is $4$, so $\Rep_2(\AGL(3,2),\rho_i)$ has minimum distance $8$ for $i=1$ or $2$.  
However, it follows from Proposition \ref{distlowerbound} that $C(\AGL(3,2),\mathcal{I})$ has minimum distance $12$, 
where $\mathcal{I}=(\rho_1,\rho_2)$.

\subsection{The Symmetric Group $S_6$.}\label{secs62}  The second action we consider is 
$S_6$ acting on the right cosets of subgroups of order $12$.  Using GAP, we 
determine that $S_6$ has four conjugacy classes of subgroups of order $12$.  Let $H_i$ be a representative for
each conjugacy class, $\Omega_i$ be the set of right cosets of $H_i$ in $S_6$, and $\rho_i$ be
the representation that describes the action.  We calculate in GAP 
the size of the supports for an element from each conjugacy class for each representation, which we
present in Table \ref{s62}.
From this table we see that $S_6$ under $\rho_3$ has the smallest minimal degree, which is $44$.
Thus, by letting $\mathcal{I}=(\rho_1,\ldots,\rho_4)$, it follows that $$\min_{\rho\in\mathcal{I}}\{\delta(\Rep_4(C(T,\rho))\}=4\cdot 44=176.$$
However, Proposition \ref{distlowerbound} implies that $C(S_6,\mathcal{I})$ has minimum distance $224$, the minimum of the sums of
the supports over the four actions, as is shown in Table \ref{s62}.  
\begin{table}[t]
\centering
\begin{tabular}{l|c|c|c|c|c|c|c|c|c|c|c}
\hline\hline
Conjugacy Class&1A&2A&2B&2C&3A&3B&4A&4B&5AB&6A&6B\\
\hline
$|\supp(t\rho_1)|$&0&56&60&60&60&48&60&60&60&60&60\\\hline
$|\supp(t\rho_2)|$&0&56&60&60&48&60&60&60&60&60&60\\\hline
$|\supp(t\rho_3)|$&0&56&44&60&57&60&60&60&60&59&60\\\hline
$|\supp(t\rho_4)|$&0&56&60&44&60&57&60&60&60&60&59\\\hline\hline
Sum of supports&0&224&224&224&225&225&240&240&240&239&239\\\hline
\end{tabular}
\caption{$S_6$ acting on subgroups of order $12$, one from each of the four conjugacy classes}
\label{s62}
\end{table} 
\begin{table}[h]
\centering
\begin{tabular}{c|c|c|c|c|c|c|c|c|c}
\hline\hline
$\mathcal{I}$&$\delta(C)$&$\mathcal{I}$&$\delta(C)$&$\mathcal{I}$&$\delta(C)$&$\mathcal{I}$&$\delta(C)$&$\mathcal{I}$&$\delta(C)$\\
\hline
$\{1,1,1,1\}$&192&$\{2,2,2,1\}$&204&$\{4,4,4,2\}$&192&$\{3,3,4,4\}$&208&$\{3,3,1,2\}$&208\\\hline
$\{2,2,2,2\}$&192&$\{2,2,2,3\}$&201&$\{4,4,4,3\}$&192&$\{1,1,2,3\}$&216&$\{3,3,1,4\}$&208\\\hline
$\{3,3,3,3\}$&176&$\{2,2,2,4\}$&204&$\{1,1,2,2\}$&216&$\{1,1,2,4\}$&213&$\{3,3,2,4\}$&208\\\hline
$\{4,4,4,4\}$&176&$\{3,3,3,1\}$&192&$\{1,1,3,3\}$&208&$\{1,1,3,4\}$&214&$\{4,4,1,2\}$&208\\\hline
$\{1,1,1,2\}$&204&$\{3,3,3,2\}$&192&$\{1,1,4,4\}$&208&$\{2,2,1,3\}$&213&$\{4,4,1,3\}$&208\\\hline
$\{1,1,1,3\}$&204&$\{3,3,3,4\}$&192&$\{2,2,3,3\}$&208&$\{2,2,1,4\}$&216&$\{4,4,2,3\}$&208\\\hline
$\{1,1,1,4\}$&201&$\{4,4,4,1\}$&192&$\{2,2,4,4\}$&208&$\{2,2,3,4\}$&213&$\{1,2,3,4\}$&224\\\hline\hline
\end{tabular}
\caption{$S_6$ acting on subgroups of order $12$, one from each of the four conjugacy classes}
\label{max}
\end{table} 

The minimum distance of $C(S_6,\mathcal{I})$ is actually maximal from an alternative perspective.  Consider
the set of representations $\mathcal{S}=\{\rho_1,\rho_2,\rho_3,\rho_4\}$.  Let $\mathcal{I}$ be formed
by taking any four elements from $\mathcal{S}$ but allowing repeats of representations.  Note that the order in which the representations
appear in $\mathcal{I}$ does not affect the minimum distance of the code it generates.  Using GAP, we 
calculate the minimum distance $\delta(C)$ of $C=C(S_6,\mathcal{I})$ for each possible $\mathcal{I}$ and give this in Table \ref{max}.  We see that 
this minimum distance is maximised when each representation appears exactly once in $\mathcal{I}$, as above.  
In Table \ref{max} we label $\mathcal{I}=(\rho_{i_1},\rho_{i_2},\rho_{i_3},\rho_{i_4})$ by $\{i_1,i_2,i_3,i_4\}$.  

\subsection{Proof of Theorem \ref{mainthm}}\label{proofmain}
Let $T$ be an abstract group, $\mathcal{I}$ be an ordered $r$-tuple of (not necessarily distinct)
permutation representations of $T$ into $S_q$ for some $q$.  
By letting $\delta_{\tw}=\delta(C(T,\mathcal{I}))$ and $\delta_{\rep}=\min_{\rho\in\mathcal{I}}\{\delta(\Rep_r(C(T,\rho))\}$, the first assertions
of Theorem \ref{mainthm} follow from Proposition \ref{distlowerbound}.  The final assertions follow from Sections \ref{secs6}, \ref{seca6},
\ref{secasl} and \ref{secs62}.  

\section{Acknowledgements}

For the first author, this research was supported by the Australian Research Council Federation Fellowship FF0776186 of the second author.


\begin{thebibliography}{10}
\providecommand{\url}[1]{{#1}}
\providecommand{\urlprefix}{URL }
\expandafter\ifx\csname urlstyle\endcsname\relax
  \providecommand{\doi}[1]{DOI~\discretionary{}{}{}#1}\else
  \providecommand{\doi}{DOI~\discretionary{}{}{}\begingroup
  \urlstyle{rm}\Url}\fi

\bibitem{bailey}
Bailey, R.F.: Error-correcting codes from permutation groups.
\newblock Discrete Math. \textbf{309}(13), 4253--4265 (2009)

\bibitem{blake}
Blake, I.: Permutation codes for discrete channels.
\newblock Information Theory, IEEE Transactions on \textbf{20}(1), 138 -- 140
  (1974)

\bibitem{blakeetal}
Blake, I.F., Cohen, G., Deza, M.: Coding with permutations.
\newblock Inform. and Control \textbf{43}(1), 1--19 (1979)

\bibitem{distreg}
Brouwer, A.E., Cohen, A.M., Neumaier, A.: Distance-regular graphs,
  \emph{Ergebnisse der Mathematik und ihrer Grenzgebiete (3) [Results in
  Mathematics and Related Areas (3)]}, vol.~18.
\newblock Springer-Verlag, Berlin (1989)

\bibitem{pcam}
Cameron, P.J.: Permutation groups, \emph{London Mathematical Society Student
  Texts}, vol.~45.
\newblock Cambridge University Press, Cambridge (1999)

\bibitem{camgad}
Cameron, P.J., Gadouleau, M.: Remoteness of permutation codes.
\newblock arXiv:1110.2028v2  (2012)

\bibitem{camwan}
Cameron, P.J., Wanless, I.M.: Covering radius for sets of permutations.
\newblock Discrete Math. \textbf{293}(1-3), 91--109 (2005)

\bibitem{chu1}
Chu, W., Colbourn, C.J., Dukes, P.: Constructions for permutation codes in
  powerline communications.
\newblock Des. Codes Cryptogr. \textbf{32}(1-3), 51--64 (2004)

\bibitem{chu2}
Chu, W., Colbourn, C.J., Dukes, P.: On constant composition codes.
\newblock Discrete Appl. Math. \textbf{154}(6), 912--929 (2006)

\bibitem{boh}
Cline, E., Parshall, B., Scott, L.: Cohomology of finite groups of {L}ie type.
  {I}.
\newblock Inst. Hautes \'Etudes Sci. Publ. Math. (45), 169--191 (1975)

\bibitem{chu3}
Colbourn, C.J., Kl{\o}ve, T., Ling, A.C.H.: Permutation arrays for powerline
  communication and mutually orthogonal {L}atin squares.
\newblock IEEE Trans. Inform. Theory \textbf{50}(6), 1289--1291 (2004)

\bibitem{at}
Conway, J.H., Curtis, R.T., Norton, S.P., Parker, R.A., Wilson, R.A.: Atlas of
  finite groups.
\newblock Oxford University Press, Eynsham (1985).
\newblock Maximal subgroups and ordinary characters for simple groups, With
  computational assistance from J. G. Thackray

\bibitem{GKK}
Coxeter, H.S.M., Moser, W.O.J.: Generators and relations for discrete groups,
  third edn.
\newblock Springer-Verlag, New York (1972).
\newblock Ergebnisse der Mathematik und ihrer Grenzgebiete, Band 14

\bibitem{dixmort}
Dixon, J.D., Mortimer, B.: Permutation groups, \emph{Graduate Texts in
  Mathematics}, vol. 163.
\newblock Springer-Verlag, New York (1996)

\bibitem{frankl}
Frankl, P., Deza, M.: On the maximum number of permutations with given maximal
  or minimal distance.
\newblock J. Combinatorial Theory Ser. A \textbf{22}(3), 352--360 (1977)

\bibitem{GAP4}
The GAP~Group: {GAP -- Groups, Algorithms, and Programming, Version 4.6.3}
  (2013).
\newblock \urlprefix\url{{http://www.gap-system.org}}

\bibitem{diagnt}
Gillespie, N., Praeger, C.: Diagonally neighbour transitive codes and frequency
  permutation arrays.
\newblock Journal of Algebraic Combinatorics  (2013).
\newblock \doi{10.1007/s10801-013-0465-6}

\bibitem{grove}
Grove, L.C.: Groups and characters.
\newblock Pure and Applied Mathematics (New York). John Wiley \& Sons Inc., New
  York (1997).
\newblock A Wiley-Interscience Publication

\bibitem{hanvinck}
Han~Vinck, A.J.: Coded modulation for power line communications.
\newblock AE\"U Journal \textbf{Jan}, 45--49 (2000).
\newblock ArXiv:1104.4528v1

\bibitem{sophie2}
Huczynska, S.: Equidistant frequency permutation arrays and related constant
  composition codes.
\newblock Des. Codes Cryptogr. \textbf{54}(2), 109--120 (2010)

\bibitem{sophie1}
Huczynska, S., Mullen, G.L.: Frequency permutation arrays.
\newblock J. Combin. Des. \textbf{14}(6), 463--478 (2006)

\bibitem{lander}
Lander, E.S.: Symmetric designs: an algebraic approach, \emph{London
  Mathematical Society Lecture Note Series}, vol.~74.
\newblock Cambridge University Press, Cambridge (1983)

\bibitem{luo}
Luo, Y., Fu, F.W., Vinck, A.J.H., Chen, W.: On constant-composition codes over
  {$Z_q$}.
\newblock IEEE Trans. Inform. Theory \textbf{49}(11), 3010--3016 (2003)

\bibitem{stateoftheart}
Pavlidou, N., Han~Vinck, A., Yazdani, J., Honary, B.: Power line
  communications: state of the art and future trends.
\newblock Communications Magazine, IEEE \textbf{41}(4), 34 -- 40 (2003)

\end{thebibliography}

\end{document}